\documentclass[10pt, leqno]{amsart}
\usepackage[latin1]{inputenc}
\usepackage{amsfonts}
\usepackage{amsmath}
\usepackage{amssymb}
\usepackage{amsthm}
\usepackage[T1]{fontenc}

\usepackage{enumerate}
\usepackage{verbatim}
\usepackage{lmodern}

\usepackage{mathtools}

\RequirePackage[dvipsnames,usenames]{color}  
\definecolor{VeryDarkGreen}{rgb}{0,0.18,0.08}
\definecolor{VeryDarkBrown}{rgb}{0.12,0.08,0.04}
\usepackage[pdftex]{hyperref}
\hypersetup{
bookmarks,
bookmarksdepth=2,
bookmarksopen,
bookmarksnumbered,
pdfstartview=FitH,
colorlinks,backref,hyperindex,
linkcolor=VeryDarkBrown,
citecolor=VeryDarkGreen,
urlcolor=VeryDarkGreen
}
\usepackage[all]{xy}
\SelectTips{cm}{10}

\DeclareMathOperator{\Supp}{Supp}

\DeclareMathOperator{\Spec}{Spec}

\newcommand{\eps}{\varepsilon}

\title[Reductions of non-lc ideals]{Reductions of non-lc ideals and non $F$-pure ideals assuming weak ordinarity}

\begin{document}

\swapnumbers
\theoremstyle{plain}
\newtheorem{Le}{Lemma}[section]
\newtheorem{Ko}[Le]{Corollary}
\newtheorem{Theo}[Le]{Theorem}
\newtheorem*{TheoB}{Theorem}
\newtheorem{Prop}[Le]{Proposition}
\newtheorem*{PropB}{Proposition}
\newtheorem{Con}[Le]{Conjecture}
\theoremstyle{definition}
\newtheorem{Def}[Le]{Definition}
\newtheorem*{DefB}{Definition}
\newtheorem{Bem}[Le]{Remark}
\newtheorem{Bsp}[Le]{Example}
\newtheorem*{BspB}{Example}
\newtheorem{Be}[Le]{Observation}
\newtheorem{Sit}[Le]{Situation}
\newtheorem{Que}[Le]{Question}
\newtheorem{Dis}[Le]{Discussion}
\newtheorem{Prob}[Le]{Problem}
\newtheorem{Konv}[Le]{Convention}
\newtheorem{claim}[Le]{Claim}
\swapnumbers
\newtheorem{ConC}{Conjecture}

\author{Axel St\"abler}

\address{Axel St\"abler\\
Johannes Gutenberg-Universit\"at Mainz\\ Fachbereich 08\\
Staudingerweg 9\\
55099 Mainz\\
Germany}

\email{staebler@uni-mainz.de}

\date{\today}


\begin{abstract}
Assume $X$ is a variety over $\mathbb{C}$, $A \subseteq \mathbb{C}$ is a finitely generated $\mathbb{Z}$-algebra and $X_A$ a model of $X$ (i.e. $X_A \times_A \mathbb{C} \cong X$). Assuming the weak ordinarity conjecture we show that there is a dense set $S \subseteq \Spec A$ such that for every closed point $s$ of $S$ the reduction of the maximal non-lc ideal filtration $\mathcal{J}'(X, \Delta, \mathfrak{a}^\lambda)$ coincides with the non-$F$-pure ideal filtration $\sigma(X_s, \Delta_s, \mathfrak{a}_s^\lambda)$ provided that $(X, \Delta)$ is klt or if $(X, \Delta)$ is log canonical, $\mathfrak{a}$ is locally principal and the non-klt locus is contained in $\mathfrak{a}$.
\end{abstract}

\maketitle

\section*{Introduction}
It has been known for several decades that there are close connections between singularities arising from the minimal model program in birational geometry over $\mathbb{C}$ and  so-called $F$-singularities in characteristic $p > 0$ which arose from tight closure theory.
To state this relationship recall the following notion: Given any variety $X$ over a field $k$ of characteristic zero and a finitely generated $\mathbb{Z}$-algebra $A \subseteq k$ we call a variety $Y$ over $\Spec A$ a \emph{model of $X$} if the base change of the generic fiber $Y_0 \times_{\Spec A} \Spec k$ is isomorphic to $X$.

For instance, it is known that Kawamata log terminal singularities (klt singularities for short) relate to $F$-regular singularities in the following way  (\cite{harawatanabe}, \cite{hararationalfrobenius}): A $\mathbb{Q}$-Gorenstein variety $X$ over $\mathbb{C}$ has klt singularities if only if for every model of $X$ there is a dense open set $U$ of $\Spec A$ such that for all closed points $u \in U$ the fiber $X_u$ has $F$-regular singularities.

Several years ago, Musta\c{t}\u{a} and Srinivas (\cite{mustatasrinivasordinary}) and Musta\c{t}\u{a} (\cite{mustataordinary}) have shown that a similar relation between \emph{multiplier ideals}, characteristic zero objects which, in particular, may be used to define the notion of klt singularities, and test ideals, characteristic $p >0 $ objects which can be used to define the notion of $F$-regularity, is true if and only if a certain natural conjecture concerning Frobenius actions on cohomology holds. Let us recall this conjecture:

\begin{ConC}
\label{weakordinarity}
Let $X$ be an $n$-dimensional smooth projective variety over a field $k$ of characteristic zero. Given a model of $X$ over a finitely generated $\mathbb{Z}$-subalgebra $A$ of $k$, there exists a Zariski-dense set of closed points $S \subseteq \Spec A$ such that the action of Frobenius on $H^n(X_s, \mathcal{O}_{X_s})$ is bijective for all $s \in S$.
\end{ConC}

This is usually referred to as the \emph{weak ordinarity conjecture} since it is strictly weaker than $X_s$ being ordinary in the sense of Bloch and Kato (\cite{blochkatopadic}).

More recently, Bhatt, Schwede and Takagi (\cite[Theorem 5.3]{bhattschwedetakagiweakordinarityfsing}) have strengthened the result of Musta\c{t}\u{a} and Srinivas as follows: They show that Conjecture \ref{weakordinarity} above is equivalent to the following

\begin{ConC}
\label{ConTauJGeneral}
Let $X$ be a normal variety over a field $k$ of characteristic zero. Suppose that $\Delta$ is a $\mathbb{Q}$-divisor such that $K_X + \Delta$ is $\mathbb{Q}$-Cartier and $\mathfrak{a}$ is a non zero ideal sheaf in $\mathcal{O}_X$. Given a model of $(X, \Delta, \mathfrak{a})$ over a finitely generated $\mathbb{Z}$-algebra $A \subseteq k$ there exists a dense set of closed points $S \subseteq \Spec A$ such that
\begin{equation}\label{Blubber2}
\tau(X_s, \Delta_s, \mathfrak{a}_s^\lambda) = \mathcal{J}(X, \Delta, \mathfrak{a}^\lambda)_s
\end{equation}
for all $s \in S$ and all $\lambda \geq 0$, where $\tau(X_s, \Delta_s, \mathfrak{a}^\lambda)$ denotes the test ideal associated to $X_s, \Delta_s, \mathfrak{a}_s, \lambda$ and similarly $\mathcal{J}(X, \Delta, \mathfrak{a}^\lambda)$ denotes the multiplier ideal sheaf associated to $X, \Delta, \mathfrak{a}, \lambda$. Furthermore, if we have finitely many triples $(X_i, \Delta_i, \mathfrak{a}_i)$ as above and corresponding models over $\Spec A$, then there is a dense subset of closed points in $\Spec A$ such that (\ref{Blubber2}) holds for all these triples.
\end{ConC}
We note that the point here is to allow $\lambda$ to be arbitrary. Namely, it is known that for fixed $\lambda$ there is an \emph{open} set of closed points of $\Spec A$ such that (\ref{Blubber2}) holds (cf. \cite[Proposition 2.11]{bhattschwedetakagiweakordinarityfsing} and references given therein).

Another important class of singularities in characteristic zero are so-called \emph{log canonical singularities}. These relate to so-called \emph{$F$-pure singularities} in the following way: If $X$ is a $\mathbb{Q}$-Gorenstein variety over $\mathbb{C}$ with a model $X_A$ such that infinitely many reductions $X_s$ are $F$-pure then $X$ is log canonical (see \cite{harawatanabe}).

The converse to this is open. However, Takagi has shown (\cite[Theorem 2.11]{takagifpurelccorrespondence}) that assuming Conjecture \ref{weakordinarity} if $X$ is $\mathbb{Q}$-Gorenstein and log canonical, then given any model over $\Spec A$ there is a dense set of closed points $S$ such that $X_s$ is $F$-pure for all $s \in S$.

In their paper, Bhatt, Takagi and Schwede raise the following question (\cite[Question 5.6]{bhattschwedetakagiweakordinarityfsing}): Does Conjecture \ref{weakordinarity} imply a similar result concerning the reduction of non-lc ideals $\mathcal{J}'(X, \Delta, \mathfrak{a}^\lambda)$ and non-$F$-pure ideals $\sigma(X_s, \Delta_s, \mathfrak{a}^\lambda_s)$?

The notion of non-lc ideal relates to log canonical singularities roughly in the same way that multiplier ideals relate to klt singularities. The same is true for non-$F$-pure ideals and $F$-pure singularities.

The goal of this note is to answer \cite[Question 5.6]{bhattschwedetakagiweakordinarityfsing} in the following cases:

\begin{TheoB}
Assume that Conjecture \ref{weakordinarity} holds. Let $X$ be a normal projective variety over a field k of characteristic zero, $\Delta$ an effective $\mathbb{Q}$-divisor such that $K_X + \Delta$ is $\mathbb{Q}$-Cartier and $\mathfrak{a}$ a non-zero ideal sheaf in $\mathcal{O}_X$. Assume one of the following conditions holds:
\begin{enumerate}[(a)]
\item{$(X, \Delta)$ is klt.}
\item{$(X, \Delta)$ is log canonical, $\mathfrak{a}$ is locally principal and the non-klt locus of $X$ is contained in $\mathfrak{a}$.}
\end{enumerate}
Then given any model of $(X, \Delta, \mathfrak{a})$ over a finitely generated $\mathbb{Z}$-subalgebra $A$ of $k$ there exists a dense set of closed points $S$ of $\Spec A$ such that for all $\lambda \geq 0$ and all $s \in S$ one has
\[ \sigma(X_s, \Delta_s, \mathfrak{a}_s^\lambda) = \mathcal{J}'(X, \Delta, \mathfrak{a}^\lambda)_s. \]
\end{TheoB}

In fact, for part (b) the condition that $(X, \Delta)$ is log canonical is only needed to ensure equality in the case $\lambda = 0$. If we only care about $\lambda > 0$ then we may drop this extra assumption.

Let us now give a brief summary of the contents of this paper.
In Section \ref{NonFpureRecap} we explain the notion of non-$F$-pure ideal and recall several technical results from \cite{staeblerassgradecartier} that we shall need. In Section \ref{maximalnonlc} we recall the notion on maximal non-lc ideals, and prove the above theorem after some preliminary work.

\subsection*{Acknowledgements} I thank Manuel Blickle, Manfred Lehn, Mircea Musta\c{t}\u{a} and Karl Schwede for useful discussions. I am also indebted to a referee for several useful suggestions. The author acknowledges support by grant STA 1478/1-1 and by SFB/Transregio 45 Bonn-Essen-Mainz of the Deutsche Forschungsgemeinschaft (DFG).

\section{A review of non-$F$-pure ideals}
\label{NonFpureRecap}
In this section we briefly review the theory of non-$F$-pure ideals (as developed in \cite[Section 4]{staeblerassgradecartier}, building on \cite{fujinotakagischwedenonlc}) and recall the definition of test ideal.

We say that a $\mathbb{Q}$-Cartier $\mathbb{Q}$-divisor $A$ has index $n$ if $nA$ is a Cartier-divisor.
Fix a normal domain $R$ of finite type over an $F$-finite field. Assume we are given an effective $\mathbb{Q}$-divisor $\Delta$ such that $K_{\Spec R} + \Delta$ is $\mathbb{Q}$-Cartier with index not divisible by $p$. If $(p^{e_0} -1)(K_{\Spec R} + \Delta)$ is Cartier, then there is a canonically associated map (up to multiplication by a unit) $\varphi_{\Delta}\colon F_\ast^{e_0} R((1-p^{e_0})(K_{\Spec R} + \Delta)) \to R$ (see \cite[Section 3]{schwedefadjunction} for details).

\begin{Def}
Let $R$ be a normal domain of finite type over an $F$-finite field, $\mathfrak{a}\subseteq R$ an ideal. Let $\Delta$ be a $\mathbb{Q}$-divisor such that $K_{\Spec R} + \Delta$ is $\mathbb{Q}$-Cartier with index not divisible by $p$. Fix a rational number $\lambda \geq 0$.
\begin{enumerate}[(a)]
\item{We define the \emph{test ideal} \[\tau(R, \Delta, \mathfrak{a}^\lambda)\] with respect to $\Delta, \mathfrak{a}$ and $\lambda$ as the smallest non-zero ideal $J$ such that \[\varphi_\Delta^{e}(F_\ast^{e\cdot e_0} \mathfrak{a}^{\lceil\lambda(p^{e\cdot e_0} -1) \rceil} J) \subseteq J\] for all $e \geq 1$}
\item{We define the \emph{non-$F$-pure ideal} \[\sigma(R, \Delta, \mathfrak{a}^\lambda)\] with respect to $\Delta, \mathfrak{a}$ and $\lambda$ as follows.
Write \[ \sigma_1 = \sum_{e \geq 1} \varphi_\Delta^e(F_\ast^{e\cdot e_0} \mathfrak{a}^{\lceil \lambda (p^{e\cdot e_0} -1) \rceil} ).\] Then inductively define \[\sigma_n = \sum_{e \geq 1} \varphi_{\Delta}^e(F_\ast^{e \cdot e_0}  \mathfrak{a}^{\lceil \lambda (p^{e\cdot e_0}-1)\rceil} \sigma_{n-1}).\] By a result of Blickle (\cite[Proposition 2.13]{blicklep-etestideale}) this descending chain stabilizes. Its stable member is by definition $\sigma(R, \Delta, \mathfrak{a}^\lambda)$.
}
\end{enumerate}
\end{Def}

We note that in \cite{fujinotakagischwedenonlc} the $\sigma_n$ are defined by taking the integral closure of the $\mathfrak{a}^{\lceil \lambda (p^e -1)\rceil}$. Denoting this variant by $\bar{\sigma}(R, \Delta, \mathfrak{a}^\lambda)$ one has an inclusion $\sigma(R, \Delta, \mathfrak{a}^\lambda) \subseteq \bar{\sigma}(R, \Delta, \mathfrak{a}^\lambda)$ which is an equality in the case $\mathfrak{a} = (f)$ since $R$ is normal. For our purposes it doesn't matter with which variant we work so we choose the simpler one (cf.\ also the discussion before Question \ref{q1}).

These constructions commute with localization so that, given a normal scheme $X$ of finite type over an $F$-finite field, a $\mathbb{Q}$-divisor $\Delta$ such that $K_X + \Delta$ is $\mathbb{Q}$-Cartier with index not divisible by $p$ and an ideal sheaf $\mathfrak{a}$ we may glue to obtain $\tau(X, \Delta, \mathfrak{a}^{\lambda})$ and $\sigma(X, \Delta, \mathfrak{a}^\lambda)$ for any $\lambda \geq 0$.

In the following we summarize the results on non-$F$-pure ideals that we shall need in the sequel.

\begin{Prop}
\label{SigmaProperties1}
Let $R$ be normal domain of finite type over an $F$-finite field, and $0 \neq f \in R$. Fix a rational number $\lambda > 0$ and let $\Delta$ be an effective $\mathbb{Q}$-divisor on $\Spec R$ such that $K_{\Spec R} + \Delta$ is $\mathbb{Q}$-Cartier with index not divisible by $p$
\begin{enumerate}[(i)]
\item{The filtration $\sigma(R, \Delta, f^\mu)_{\mu \geq 0}$ is discrete and non-increasing.}
\item{Assume that the non-$F$-regular locus of $R$ is contained in $(f)$. If $\lambda \in \mathbb{Z}_{(p)}$, then $\sigma(R, \Delta, f^{\lambda}) = \tau(R, \Delta, f^{\lambda - \eps})$ for all $0 < \eps \ll 1$. If the denominator of $\lambda$ is divisible by $p$, then $\sigma(R, \Delta, f^{\lambda}) = \tau(R, \Delta, f^{\lambda + \eps})$ for all $0 < \eps \ll 1$.}
\item{Assume that the non-$F$-regular locus of $R$ is contained in $(f)$. If the denominator of $\lambda$ is divisible by $p$, then $\sigma(R, \Delta, f^{\lambda}) = \sigma(R, \Delta, f^{\lambda + \eps})$ for all $0 < \eps \ll 1$.}
\item{For $\lambda \geq 1$ and  any ideal $\mathfrak{a} \subseteq R$ we have $\mathfrak{a} \cdot \sigma(R, \Delta, \mathfrak{a}^{\lambda -1}) \subseteq \sigma(R, \Delta, \mathfrak{a})$. If $\mathfrak{a} =(f)$ and the non-$F$-regular locus of $R$ is contained in $(f)$, then equality holds for $\lambda > 1$.}
\end{enumerate}
\end{Prop}
\begin{proof}
(i) is \cite[Proposition 4.6]{staeblerassgradecartier}, (ii) is \cite[Proposition 4.9, Corollary 4.16]{staeblerassgradecartier}, (iii) is \cite[Proposition 4.14]{staeblerassgradecartier} and (iv) is \cite[Proposition 4.20]{staeblerassgradecartier}, the addendum follows from the Skoda for test ideals (using (ii)).
\end{proof}

\begin{Le}
\label{SigmaInclusion}
Let $R$ be $F$-finite, $\Delta$ a $\mathbb{Q}$-divisor such that $(p^{e_0} -1)(K_{\Spec R} + \Delta)$ is Cartier, $\mathfrak{a}$ an ideal and $\lambda \geq 0$. If $g \in \mathfrak{a}^m$, then $\sigma(R, \Delta, g^{\frac{\lambda}{m}}) \subseteq \sigma(R, \Delta, \mathfrak{a}^{\lambda})$.
\end{Le}
\begin{proof}
Observe that $m\lceil \frac{\lambda}{m}(p^{ea} -1) \rceil \geq \lceil \lambda(p^{ea} -1) \rceil$. Thus for $e \geq 1$ we have
\[\varphi_\Delta^e(F_\ast^{e\cdot e_0} (g)^{\lceil\frac{\lambda}{m} (p^{e \cdot e_0} -1)\rceil}) \subseteq \varphi^e_\Delta(F_\ast^{e \cdot e_0} \mathfrak{a}^{\lceil\lambda (p^{e \cdot e_0} -1)\rceil}).\] This shows that $\sigma_n(R, \Delta, g^{\frac{\lambda}{m}}) \subseteq \sigma_n(R, \Delta, \mathfrak{a}^\lambda)$ for all $n$.
\end{proof}

\section{maximal non-lc ideals and non-$F$-pure ideals}
\label{maximalnonlc}

The goal of this section is to prove that the weak ordinarity conjecture (Conjecture \ref{weakordinarity} from the introduction) implies another conjecture that claims that (given any model) the reduction of maximal non-lc ideals coincides with non-$F$-pure ideals on a dense set (Conjecture \ref{con2} below).

We recall the definition of maximal non-lc ideal sheaves and multiplier ideal sheaves. Given an effective divisor $D$ we write ${}^kD = \sum E_i$, where the $E_i$ are the prime divisors which occur with multiplicity exactly $k$ in $D$.

By a \emph{variety} we mean a separated integral scheme of finite type over a field. \emph{Unless explicitly stated to the contrary we will from now on  work over a field $k$ of characteristic zero}.

\begin{Def}
Let $X$ be a normal variety and $\Delta$ an effective $\mathbb{Q}$-divisor such that $K_X + \Delta$ is $\mathbb{Q}$-Cartier and let $\mathfrak{a}$ be an ideal sheaf. Let $f: Y \to X$ be a resolution such that $K_Y + \Delta_Y = f^\ast(K_X + \Delta)$ and $ \mathfrak{a} \mathcal{O}_Y = \mathcal{O}(-E)$ , where $\Supp \Delta_Y \cup \Supp E$ is snc.

Then we define the \emph{maximal non-lc ideal sheaf} (associated to $\Delta$ and $\mathfrak{a}^c$), as \[\mathcal{J}'(X, \Delta, \mathfrak{a}^c) = f_\ast \mathcal{O}_Y(- \lfloor \Delta_Y + c E\rfloor  + \sum_{k=1}^\infty {}^k(\Delta_Y + cE)).\]

Similarly, we define the \emph{multiplier ideal sheaf} as \[ \mathcal{J}(X, \Delta, \mathfrak{a}^c) = f_\ast \mathcal{O}_Y(\lceil - \Delta_Y - cE \rceil).\]
\end{Def}

Given a birational morphism $f: Y \to X$ we denote the strict transform of a divisor $D$ on $X$ by $f_\ast^{-1} D$.

\begin{Def}
Let $X$ be a normal variety and $\Delta$ an effective $\mathbb{Q}$-divisor such that $K_X + \Delta$ is $\mathbb{Q}$-Cartier. Fix a resolution $f: Y \to X$ such that $\Delta_Y$ has snc support, where $K_Y + \Delta_Y = f^\ast(K_X + \Delta)$. Write $K_Y + f_\ast^{-1} \Delta = f^\ast(K_X + \Delta) + \sum_E a_E E$, where the $E$ are distinct exceptional prime divisors, and $\Delta = \sum d_i \Delta_i$, where the $\Delta_i$ are distinct prime divisors.
\begin{enumerate}[(a)]
\item{The pair $(X, \Delta)$ is called \emph{Kawamata log terminal($=$klt)}, if $\lfloor \Delta \rfloor = 0$ and $a_E > -1$. Equivalently, $\mathcal{J}(X, \Delta) = \mathcal{O}_X$.}
\item{The pair $(X, \Delta)$ is called \emph{purely log terminal($=$plt)}, if $a_E > -1$ and $d_i \leq 1$}
\item{The pair $(X, \Delta)$ is called \emph{log canonical}, if $a_E \geq -1$ and $d_i \leq 1$.}
\item{We define the \emph{non-klt locus} of $(X, \Delta)$ as the closed subscheme $V(\mathcal{J}(X, \Delta))$. By abuse of notation we will say that an ideal sheaf $\mathfrak{a}$ \emph{contains the non-klt locus} if $\mathcal{J}(X, \Delta) \subseteq \mathfrak{a}$.}
\item{We define the \emph{non plt-locus} of $(X, \Delta)$ as the closed subscheme $V(f_\ast \mathcal{O}_Y(\lceil - \Delta_Y + f_\ast^{-1} \Delta \rceil))$, that is, we throw away the non-exceptional parts of $\Delta_Y$. By abuse of notation we will say that an ideal sheaf $\mathfrak{a}$ \emph{contains the non-plt locus} if $f_\ast \mathcal{O}_Y(\lceil - \Delta_Y + f_\ast^{-1} \Delta \rceil) \subseteq \mathfrak{a}$.}
\end{enumerate}
\end{Def}

Note that in particular, if $D$ is a prime divisor whose coefficient in $\Delta_Y$ is $\geq 1$, then $\mathcal{O}(-D)$ is contained in the non-klt locus.

\begin{Le}
\label{maximalnonlicforidealseps}
Let $X$ be a normal variety, $\Delta$ an effective $\mathbb{Q}$-divisor such that $K_X + \Delta$ is $\mathbb{Q}$-Cartier, $\mathfrak{a}$ an ideal sheaf and $\lambda \in \mathbb{Q}_{\geq 0}$. If $f: Y \to X$ is a log resolution of $(X, \Delta, \mathfrak{a})$, then for all $0 < \eps \ll 1$ \[\mathcal{J}'(X, \Delta, \mathfrak{a}^\lambda) = f_\ast \mathcal{O}_Y(\lceil - \Delta_Y - \lambda F + \eps(\Delta_Y + \lambda F)\rceil),\] where $f^{-1}(\mathfrak{a}) = \mathcal{O}_Y(-F)$. 
\end{Le}
\begin{proof}
Let $D$ be a prime divisor occurring in $\Delta_Y + \lambda F$. Denote its coefficient in $-\Delta_Y - \lambda F$ by $a_D$. Then if $a_D \notin \mathbb{Z}$ the assertion is clear for $0 < \eps \ll 1$. If $a_D \in \mathbb{Z}$ one distinguishes the cases $a_D \geq 0$ and $a_D < 0$.
\end{proof}

\begin{Prop}
\label{Maximalnonlcmultklt}
Let $X$ be a normal variety, $\Delta$ an effective $\mathbb{Q}$-divisor such that $K_X + \Delta$ is $\mathbb{Q}$-Cartier, $\mathfrak{a}$ an ideal sheaf and $\lambda \in \mathbb{Q}_{> 0}$. Assume that $\mathfrak{a}$ contains the non-klt locus of $(X, \Delta)$. Then \[\mathcal{J}'(X, \Delta, \mathfrak{a}^{\lambda}) = \mathcal{J}(X, \Delta, \mathfrak{a}^{\lambda - \eps}).\]
\end{Prop}


\begin{proof}
Fix a log resolution $f: Y \to X$ of $(X, \Delta, \mathfrak{a})$ and fix a representative of $K_X$. We will assume that $K_Y$ and $f^\ast K_X$ restricted to the regular locus coincide. Let $f^{-1}(\mathfrak{a}) = \mathcal{O}_Y(-F)$. We set $\Delta_Y = f^\ast(K_X + \Delta) - K_Y$. In particular, we have $\Delta_Y = f_\ast^{-1} \Delta - \sum_E a_E E$, where the $E$ are exceptional divisors.

Now we want to argue that \begin{align*}\mathcal{J}(X, \Delta, \mathfrak{a}^{\lambda - \eps} ) \overset{\text{def}}{=}&  f_\ast \mathcal{O}_Y(\lceil -\Delta_Y - (\lambda - \eps) F \rceil)\\\overset{!}{=}& f_\ast \mathcal{O}_Y(\lceil - \Delta_Y + \eps \Delta_Y - \lambda F + \eps \lambda F\rceil) = \mathcal{J}'(X, \Delta, \mathfrak{a}^\lambda),\end{align*} where the last equality is due to Lemma \ref{maximalnonlicforidealseps}.

Write \[A := \lceil -\Delta_Y - (\lambda - \eps)F \rceil\] and \[B := \lceil - \Delta_Y + \eps \Delta_Y - (\lambda -\eps \lambda)F \rceil.\] We write $F_{\mathrm{ne}}$ for the non-exceptional part of $F$. Note that the non-exceptional part of $A$ is negative. First we show that the non-exceptional parts of $A$ and $B$ coincide. If a prime divisor $P$ occuring in the non-exceptional part of $A$ also occurs in $F_{\mathrm{ne}}$, then it is already perturbed by $\eps F_{\mathrm{ne}}$ and the expressions coincide. If $P$ doesn't occur in $F_{\mathrm{ne}}$, then by our assumption the coefficient $a$ of $P$ in $-f_\ast^{-1} \Delta$ satisfies $\lceil a \rceil = 0$. But then also $\lceil a - \eps a \rceil = 0$ which is the corresponding component in $B$.

\begin{claim}
The negative exceptional parts of $A$ and $B$ coincide.
\end{claim}
\begin{proof}[Proof of claim.]
Let $D$ be an exceptional divisor occurring in $A$. If $D$ does not occur in the exceptional part of $F$, then $a_D > -1$ by our assumption on the klt-locus. Hence, the coefficient of $D$ in $A$ is $\lceil a_D \rceil \geq 0$. Likewise the coefficient of $D$ in $B$ is of the form $\lceil (1 - \eps)a_D \rceil \geq 0$.

If $D$ does occur in the exceptional part of $F$, with coefficient $c_D \geq 0$ say, then the coefficient of $D$ in $A$ is of the form
\[\lceil a_D -(\lambda - \eps) c_D\rceil = \lceil a_D - \lambda c_D + \eps c_D \rceil. \] The coefficient of $D$ in $B$ is of the form
\[\lceil(1 -\eps)a_D -(\lambda - \eps \lambda) c_D\rceil = \lceil a_D - \lambda c_D - \eps(a_D - \lambda c_D)\rceil. \]
These two expressions clearly coincide if $a_D - \lambda c_D < 0$. If this is not the case, then $a_D - \lambda c_D \geq 0$ so that $D$ occurs as an effective exceptional divisor.
\end{proof}
It follows that $f_\ast \mathcal{O}_Y(A) = f_\ast \mathcal{O}_Y(B)$.
\end{proof}

\begin{Prop}
\label{Nonlcmultiplieridealrelation}
Let $X$ be a normal variety, $\Delta$ an effective $\mathbb{Q}$-divisor such that $K_X + \Delta$ is $\mathbb{Q}$-Cartier, $\mathfrak{a}$ an ideal sheaf and $\lambda \in \mathbb{Q}_{> 0}$. Assume that $\mathfrak{a}$ contains the non-plt locus of $(X, \Delta)$. Then \[\mathcal{J}'(X, \Delta, \mathfrak{a}^{\lambda}) = \mathcal{J}(X, \Delta, \mathfrak{a}^{\lambda - \eps}\mathcal{O}(-\Delta)^{-\underline{a}\eps}),\] where if $\Delta = \sum_i a_i \Delta_i$ with prime divisors $\Delta_i$ then $\mathcal{O}(-\Delta)^{-\underline{a}\eps}$ is shorthand for $\mathcal{O}(-\Delta_1)^{-a_1 \eps} \cdots \mathcal{O}(-\Delta_n)^{-a_n\eps}$.
\end{Prop}
\begin{proof}
This is similar to Proposition \ref{Maximalnonlcmultklt}. Fix a log resolution $f: Y \to X$ of $(X, \Delta, \mathfrak{a}, \mathcal{O}(-\Delta))$ and fix a representative of $K_X$. We will assume that $K_Y$ and $f^\ast K_X$ restricted to the regular locus coincide. Let $f^{-1}(\mathfrak{a}) = \mathcal{O}_Y(-F)$ and $f^{-1}(\mathcal{O}(-\Delta_i) = \mathcal{O}(-G_i)$.
We set $\Delta_Y = f^\ast(K_X + \Delta) - K_Y$. In particular, we have $\Delta_Y = f_\ast^{-1} \Delta - \sum_E a_E E$, where the $E$ are exceptional divisors.

Now we want to argue that \begin{align*}\mathcal{J}(X, \Delta, \mathfrak{a}^{\lambda - \eps} \mathcal{O}(-\Delta)^{-\underline{a}\eps}) \overset{\text{def}}{=}&  f_\ast \mathcal{O}_Y(\lceil -\Delta_Y - (\lambda - \eps) F  + \eps \sum_i a_i G_i\rceil)\\\overset{!}{=}& f_\ast \mathcal{O}_Y(\lceil - \Delta_Y + \eps \Delta_Y - \lambda F + \eps \lambda F\rceil) = \mathcal{J}'(X, \Delta, \mathfrak{a}^\lambda),\end{align*} where the last equality is due to Lemma \ref{maximalnonlicforidealseps}.

Since $\sum_i a_i G_i \vert_{f^{-1}(U)} = \Delta\vert_{f^{-1}(U)}$, where $U$ is the non-exceptional locus of $f$, we may write $G := \sum_i a_i G_i = f^{-1}_\ast \Delta + \text{Exceptional divisors}$. Since $G$ and $\Delta$ are effective the exceptional part is also effective.

Write \[A := \lceil -\Delta_Y - (\lambda - \eps) F  + \eps \sum_i a_i G_i\rceil\] and note that its non-exceptional part is negative (by choosing $0 < \eps \ll 1$ small enough). Similarly, we write \[B := \lceil - \Delta_Y + \eps \Delta_Y - \lambda F + \eps \lambda F\rceil.\] First we assert that the non-exceptional parts of $A$ and $B$ coincide. Indeed, if we denote the non-exceptional part of $F$ by $F_{\mathrm{ne}}$, then the non-exceptional part of $A$ is given by $\lceil - f_\ast^{-1} \Delta - (\lambda - \eps) F_{\mathrm{ne}} + \eps f_\ast^{-1} \Delta\rceil$ and the non-exceptional part of $B$ is given by $\lceil (\eps -1) f_\ast^{-1} \Delta - \lambda(1 - \eps) F_{\mathrm{ne}}\rceil$ -- clearly these coincide for $0 < \eps \ll 1$.

One now argues similarly to Proposition \ref{Maximalnonlcmultklt} that the negative exceptional parts of $A$ and $B$ coincide.
\end{proof}


\begin{Ko}
\label{nonlcproperties}
With the assumptions of Proposition \ref{Nonlcmultiplieridealrelation} the filtration $\mathcal{J}'(X, \Delta, \mathfrak{a}^t)$ for $t > 0$ is discrete, left-continuous and if $m$ is the minimal number of generators of $\mathfrak{a}$, then for $t > m$ we have $\mathcal{J}'(X, \Delta, \mathfrak{a}^t) = \mathfrak{a} \,\,\mathcal{J'}(X, \Delta, \mathfrak{a}^{t-1})$.
\end{Ko}
\begin{proof}
These properties follow from the corresponding properties for multiplier ideals (see \cite[Section 3.1]{mustatasrinivasordinary} for the first two properties and \cite[Corollay 1.4]{MR2492466} for the last one). 
\end{proof}

\begin{Le}
Let $X$ be a normal variety and $\Delta$ an effective $\mathbb{Q}$-divisor such that $K_X + \Delta$ is $\mathbb{Q}$-Cartier. If $D$ is an effective Cartier divisor then $\mathcal{J}'(X, \Delta, \lambda D) = \mathcal{O}(-D)\cdot \mathcal{J}'(X, \Delta, (\lambda - 1) D)$ for all $\lambda > 1$.
\end{Le}
\begin{proof}
Let $f: Y \to X$ be a log resolution of $(X, \Delta, \mathcal{O}(-D))$.
Since $D$ is Cartier $\mathcal{O}_X(-D)$ is invertible. so by Lemma \ref{maximalnonlicforidealseps} and projection formula \begin{align*}\mathcal{J}'(X, \Delta, \lambda D) &= f_\ast \mathcal{O}_Y(\lceil -\Delta_Y - \lambda f^\ast D + \eps(\Delta_Y + \lambda f^\ast D)\rceil)\\ &= f_\ast \mathcal{O}_Y(\lceil - \Delta_Y - (\lambda -1) f^\ast D +\eps(\Delta_Y + \lambda f^\ast D) \rceil - f^\ast D)\\ &= f_\ast \mathcal{O}_Y(\lceil - \Delta_Y - (\lambda -1) f^\ast D + \eps(\Delta_Y + \lambda f^\ast D) \rceil) \otimes \mathcal{O}_X(-D)\\ &= \mathcal{J}'(X, \Delta, (\lambda-1)D) \cdot \mathcal{O}_X(-D).\end{align*}
\end{proof}

Recall that given a variety $X$ over a field $k$ containing $\mathbb{Q}$ a model of $X$ over a finitely generated $\mathbb{Z}$-algebra $A$ is a vartiety $X_A$ over $\Spec A$ such that $X_A \times k \cong X$.

\begin{Con}
\label{con2}
Let $X$ be a normal variety over a field k of characteristic zero, $\Delta$ an effective $\mathbb{Q}$-divisor such that $K_X + \Delta$ is $\mathbb{Q}$-Cartier and $\mathfrak{a}$ an ideal sheaf in $\mathcal{O}_X$. Given any model of $(X, \Delta, \mathfrak{a})$ over a finitely generated $\mathbb{Z}$-subalgebra $A$ of $k$ there exists a dense set of closed points $S$ of $\Spec A$ such that for all $\lambda \geq 0$ and all $s \in S$ one has
\[ \sigma(X_s, \Delta_s, \mathfrak{a}_s^\lambda) = \mathcal{J}'(X, \Delta, \mathfrak{a}^\lambda)_s. \]
\end{Con}

We have the following

\begin{Theo}
\label{AnotherResult}
Conjecture \ref{weakordinarity} implies Conjecture \ref{con2} under the assumption that $(X, \Delta)$ is klt or if $\mathfrak{a}$ is locally principal, $(X, \Delta)$ is log canonical and the non-klt locus is contained in $(f)$.
\end{Theo}
\begin{proof}
Since the formation of multiplier ideals commutes with field base change and since reduction to positive characteristic does not differentiate between $X$ and $X \times_k \bar{k}$ we may assume that $k$ is algebraically closed.

In proving Conjecture \ref{con2} we may look at a finite open affine cover $U_i$ of $X$ and show \[\sigma(U_{i,s}, {\Delta\vert_{U_i}}_s, {\mathfrak{a}\vert_{U_i}}_s^\lambda) = \mathcal{J}'(U_{i}, {\Delta\vert_{U_i}}, {\mathfrak{a}\vert_{U_i}}^\lambda)_s \] for all $i$ simultaneously. Hence, we may assume that $X$ is affine. Since Conjecture \ref{weakordinarity} implies Conjecture \ref{ConTauJGeneral} we have a dense set $S \subset \Spec A$, where \[\mathcal{J}(X, \Delta, \mathfrak{a}^\lambda)_s = \tau(X_s, \Delta_s ,\mathfrak{a}_s^\lambda)\] for all $s \in S$. Note that this set $S$ is obtained from a dense set $S'$ guaranteed by \cite[Theorem 5.10]{mustatasrinivasordinary} by removing a closed subset. Removing a closed set of points in $S$ we may also assume that the index of $\Delta$ is coprime to the characteristics $p(s)$ of the residue fields $k(s)$.

By \cite[Theorem 15.2]{fujinotakagischwedenonlc} it suffices to show the inclusion \begin{equation}\label{eq.NeedtoShow} \mathcal{J}'(X, \Delta, \mathfrak{a}^\lambda)_s \subseteq \sigma(X_s, \Delta_s, \mathfrak{a}_s^\lambda).\end{equation} In doing this we will first reduce to the case that $\mathfrak{a} =(f)$ is principal. This will be similar to \cite[Proposition 4.3]{mustatasrinivasordinary}.

Let $h_1, \ldots, h_m$ be generators of $\mathfrak{a}$. Then for $\lambda > m$ we have
\[\mathcal{J}'(X, \Delta, \mathfrak{a}^\lambda)_s = \mathfrak{a}_s \cdot \mathcal{J}'(X, \Delta, \mathfrak{a}^{\lambda -1})_s\] by Corollary \ref{nonlcproperties} and similarly, by Proposition \ref{SigmaProperties1} (iv), we have for $\lambda \geq 1$
\[\mathfrak{a}_s \cdot \sigma(X_s, \Delta_s, \mathfrak{a}_s^{\lambda -1}) \subseteq \sigma(X_s, \Delta_s, \mathfrak{a}_s^\lambda) \]
Hence it suffices to show (\ref{eq.NeedtoShow}) for $\lambda \leq m$.

Let $g_1, \ldots, g_m$ be general linear combinations of the $h_i$ with coefficients in $k$, and let $g = g_1 \cdots g_m$. Then applying Proposition \ref{Maximalnonlcmultklt} (first and last equality) and \cite[Proposition 3.5]{mustatasrinivasordinary} we have
\[\mathcal{J}'(X, \Delta, \mathfrak{a}^\lambda) = \mathcal{J}(X, \Delta, \mathfrak{a}^{\lambda - \eps}) = \mathcal{J}(X, \Delta, g^{\frac{\lambda - \eps}{m}}) = \mathcal{J}'(X, \Delta, g^{\frac{\lambda}{m}})\] for $\lambda \leq m$. 
Assuming that (\ref{eq.NeedtoShow}) holds for principal ideals we then get for any $s$
\[ \mathcal{J}'(X, \Delta, \mathfrak{a}^\lambda)_s = \mathcal{J}'(X, \Delta, g^{\frac{\lambda}{m}})_s  \subseteq \sigma(X_s, \Delta_s, g_s^{\frac{\lambda}{m}}) \subseteq \sigma(X_s, \Delta_s, \mathfrak{a}_s^\lambda)\] where the last inclusion is due to Lemma \ref{SigmaInclusion}.

Having reduced to the case that $\mathfrak{a} = (f)$ we note that if $(X, \Delta)$ is klt, then the non-klt locus is in particular contained in $\mathfrak{a}$. Thus we just need to prove the second case.

For $\lambda = 0$ we use \cite[Theorem 2.11]{takagifpurelccorrespondence} which asserts that Conjecture \ref{weakordinarity} implies that $(X_s, \Delta_s)$ is $F$-pure for all $t \in T$, where $T$ is a dense set of closed points (note that this set $T$ is be constructed from \cite[Theorem 5.10]{mustatasrinivasordinary} by removing a closed subset. Thus we can arrange this to hold for our fixed dense set $S$ by removing a closed subset).
For $\lambda > 0 $ note that by \cite[Theorem 3.2]{takagimultiplieridealsviatightclosure}, for almost all $s$ the non $F$-regular locus of $X_s$ is contained in $(f_s)$.

Let $s \in S$. If the denominator of $\lambda$ is not divisible by $p(s)$, then by Proposition \ref{SigmaProperties1} (ii) for all $0 < \eps \ll 1$
\[\sigma(X_s, \Delta_s, f_s^\lambda) = \tau(X_s, \Delta_s, f_s^{\lambda - \eps}) \] and similarly by Proposition \ref{Maximalnonlcmultklt}
\[\mathcal{J}'(X, \Delta, f^\lambda)_s = \mathcal{J}(X, \Delta, f^{\lambda - \eps})_s.\] Since Conjecture \ref{weakordinarity} implies Conjecture \ref{ConTauJGeneral} we obtain $\sigma(X_s, \Delta_s, f_s^\lambda) = \mathcal{J}'(X, \Delta, f^\lambda)_s$ in this case.

Fix $0 < \eps \ll 1$. For all $0 < \delta \leq \eps$ there are only finitely many primes dividing denominators of $\lambda$ for which \[\mathcal{J}'(X, \Delta, f^{\lambda}) \neq \mathcal{J}'(X, \Delta, f^{\lambda + \delta}).\]
After removing a closed subset from $S$ we may assume that these primes do not occur as residue characteristics for points in $S$. Assume now that the denominator of $\lambda$ is divisible by $p(s)$ for some $s \in S$. Then we have
\[\mathcal{J}'(X, \Delta, f^{\lambda})_s = \mathcal{J}'(X, \Delta, f^{\lambda + \delta})_s, \] where we may choose $0 < \delta \ll \eps$ in such a way that the denominator of $\lambda + \delta$ is not divisible by $p(s)$. Likewise, using Proposition \ref{SigmaProperties1} (iii) we have $\sigma(X_s, \Delta_s, f_s^{\lambda}) = \sigma(X_s, \Delta_s, f_s^{\lambda + \delta})$. Putting these together we obtain
\[\mathcal{J}'(X, \Delta, f^{\lambda})_s = \mathcal{J}'(X, \Delta, f^{\lambda + \delta})_s = \sigma(X_s, \Delta_s, f_s^{\lambda + \delta}) \subseteq \sigma(X_s, \Delta_s, f_s^{\lambda}),\] where the last equality is due to the case where the denominator of $\lambda$ is not divisible by $p(s)$ and the inclusion follows since the filtration $\sigma$ is non-increasing. Since the other inclusion always holds by \cite[Theorem 15.2]{fujinotakagischwedenonlc} the proof is complete.
\end{proof}

\begin{Bem}
Our proof actually shows that for arbitrary normal $X$, if the non-klt locus of $(X, \Delta)$ is contained in $\mathfrak{a} = (f)$ then for any $\lambda > 0$ and all $s \in S$ we have an equality \[\mathcal{J}'(X, \Delta, f^\lambda)_s = \tau(X_s, \Delta_s, f_s^\lambda).\] The inclusion $\sigma(X, \Delta) \subseteq \mathcal{J}'(X, \Delta)_s$ (due to \cite[Theorem 15.2]{fujinotakagischwedenonlc}) and \cite[Theorem 2.11]{takagifpurelccorrespondence} further show that $\mathcal{J}'(X, \Delta)_s$ and $\sigma(X, \Delta)$ agree up to radical. But we cannot apply \cite[Theorem 5.10]{mustatasrinivasordinary} to ensure $F$-purity of $\mathcal{J}'(X, \Delta)_s$ in general.
\end{Bem}

\begin{Prop}
\label{ConverseConjecture}
Conjecture \ref{con2} implies Conjecture \ref{weakordinarity}.
\end{Prop}
\begin{proof}
This follows along the lines of \cite[Theorem 1.3]{mustataordinary}. Indeed, Musta\c{t}\u{a} shows that for a smooth variety $X$ over a field of characteristic zero the reductions along a model are ordinary on a dense set $S$  if the following is satisfied: For a certain element $h \in R= k[x_0, \ldots, x_{N+1}]$ of degree $= 2r \geq N +1$, where $k$ is a field of characteristic zero and \[\mathcal{J}(\mathbb{A}^{N+1}_k, h^{\lambda}) =\begin{cases} R,&  \text{ if } 0 \leq \lambda < \frac{N+1}{2r}\\
(x_0,\ldots, x_N)^{\lfloor 2\lambda r \rfloor - N},& \text{ if } \frac{N+1}{2r} \leq \lambda < 1. \end{cases}\] one needs to have $(x_0, \ldots, x_N)^{2r - N -1} \subseteq \tau(\mathbb{A}^{N+1}_{k(s)}, h^\mu_s)$ for $\mu = 1 - \frac{1}{p(s)}$.

But $\mathbb{A}^{N+1}_k$ is smooth so that by Proposition \ref{Maximalnonlcmultklt} (or \cite{fujinotakagischwedenonlc}) we obtain \[\mathcal{J}(\mathbb{A}^{N+1}_k, \mathfrak{a}^{\lambda}) = \mathcal{J}'(\mathbb{A}^{N+1}_k, \mathfrak{a}^{\lambda - \eps})\] for any $\lambda > 0$ and any $0 < \eps \ll 1$. Now Conjecture \ref{con2} ensures that \[\mathcal{J}'(\mathbb{A}^{N+1}_k, h^{\lambda - \eps})_s = \sigma(\mathbb{A}^{N+1}_{k(s)}, h_s^{\lambda - \eps})\] for all $\lambda$ and all $s \in S$, where $S$ is a dense set of closed points of any model. Fix any $\lambda$ in $[1- \frac{1}{2r}, 1)$. Then we have \[\tau(\mathbb{A}^{N+1}_{k(s)}, h_s^{\lambda}) = \sigma(\mathbb{A}^{N+1}_{k(s)}, h_s^{\lambda - \eps})\] by Proposition \ref{SigmaProperties1} (ii) for all $s$ such that $p(s)$ does not divide the denominator of $\lambda$. Hence, we obtain the desired result after excluding the closed set of $S$ for which $p(s)$ does divide the denominator.
\end{proof}

\begin{Que}
Does a Skoda theorem hold for  $\mathcal{J}'(X, \Delta, \mathfrak{a}^\lambda)$? That is, if $\mathfrak{a}$ is generated by $m$ elements, is it true that $\mathfrak{a}\,\, \mathcal{J}'(X, \Delta, \mathfrak{a}^{\lambda -1}) = \mathcal{J}'(X,\Delta, \mathfrak{a}^{\lambda})$ for $\lambda > m$?
\end{Que}

Recall that we defined $\sigma(R, \Delta, \mathfrak{a}^\lambda)$ using the $\mathfrak{a}^{\lceil \lambda(p^e -1)\rceil}$ whereas in \cite{fujinotakagischwedenonlc} the non-$F$-pure ideal $\bar{\sigma}(R, \Delta, \mathfrak{a}^\lambda)$ is defined by considering $\overline{\mathfrak{a}^{\lceil \lambda (p^{e} -1)\rceil}}$ instead. We note that Proposition \ref{SigmaProperties1} (iv) also holds for $\bar{\sigma}$. Hence, if $R$ is normal, then Theorem \ref{AnotherResult} also holds for this variant. In particular, assuming weak ordinarity $\sigma(R, \Delta, \mathfrak{a}^\lambda) = \bar{\sigma}(R, \Delta, \mathfrak{a}^\lambda)$ holds in case of a reduction on a dense set provided that $(R, \Delta)$ is klt.

At the same time it is shown in \cite[Remark 2.3 (2)]{takagiwatanabeonfpurethresholds} that $\sigma(R, \mathfrak{a}^{\lambda}) \neq \bar{\sigma}(R, \mathfrak{a}^\lambda)$ for $R = \mathbb{F}_2[x,y,z]$, $\mathfrak{a} = (x^2, y^2,z^2)$, $\lambda = \frac{3}{2}$ (and $\Delta = 0$).

\begin{Que}
\label{q1}
Are there examples where $\sigma(R, \Delta, \mathfrak{a}^{\lambda}) \neq \bar{\sigma}(R, \Delta, \mathfrak{a}^\lambda)$, where the denominator of $\lambda$ is not divisible by $p$? Given a reduction mod $p$-situation are there infinitely many primes where equality does not hold?
\end{Que}

Given the results of \cite{cantonhswbehaviorofsingsatfpt} one should probably expect the answer to the first question to be negative.

\begin{Que}
Does Conjecture \ref{weakordinarity} imply Conjecture \ref{con2} if we assume that $(X, \Delta)$ is plt?
\end{Que}

Note that by \cite[Corollary 5.4]{takagicharpadjointideals} the reduction $(X_s, \Delta_s)$ is purely $F$-regular for almost all primes $s \in S$ (see \cite[Definition 2.1]{harawatanabe} for a definition\footnote{Note that the notion is called \emph{divisorially $F$-regular} there.}). The issue is that in proving a positive characteristic analogue of Proposition \ref{Nonlcmultiplieridealrelation} we need to deal with the even locally non principal ideal $\mathcal{O}_X(-\Delta)$. At present we do not know how to handle this.

\bibliography{bibliothek.bib}
\bibliographystyle{amsalpha}
\end{document}